\documentclass[a4paper,11pt,centertags]{amsart}
\usepackage[latin1]{inputenc}
\usepackage[english]{babel}
\usepackage{cite}
\usepackage{verbatim}
\usepackage{color}
\usepackage{hyperref}
\usepackage{extarrows}
\usepackage{amsmath,amsthm}
\usepackage{comment}

\usepackage[top=3.2cm, bottom=3.2cm, left=2.8cm,right=2.8cm]{geometry}

\newtheorem{theo}{Theorem}[section]
\newtheorem{lemma}{Lemma}[section]
\newtheorem{prop}{Proposition}[section]

\theoremstyle{definition}
\newtheorem{definiz}{Definition}[section]
\newtheorem{rem}{Remark}[section]

\numberwithin{equation}{section}

\newcommand{\R}{\mathbb R}
\newcommand{\Rn}{\mathbb R^n}

\newcommand{\f}{\mathcal F}
\newcommand{\de}{\partial}
\newcommand{\eps}{\varepsilon}

\newcommand{\utilde}{\tilde u_{p}}
\newcommand{\ds}{\displaystyle}
\DeclareMathOperator{\divergenza}{div}

\begin{document}
\title[Faber-Krahn inequality for anisotropic Robin problems]{Faber-Krahn inequality for anisotropic eigenvalue problems with Robin boundary conditions} 
\author[F. Della Pietra, N. Gavitone]{
  Francesco Della Pietra and Nunzia Gavitone
}
\address{Francesco Della Pietra \\
Universit\`a degli studi di Napoli ``Federico II''\\
Dipartimento di Ma\-te\-ma\-ti\-ca e Applicazioni ``R. Caccioppoli''\\
Complesso di Monte Sant'Angelo, Via Cintia,
80126 Napoli, Italia.
}
\email{f.dellapietra@unina.it}

 \address{
Nunzia Gavitone \\
Universit\`a degli studi di Napoli ``Federico II''\\
Dipartimento di Ma\-te\-ma\-ti\-ca e Applicazioni ``R. Caccioppoli''\\
Complesso di Monte Sant'Angelo, Via Cintia,
80126 Napoli, Italia.
}
\email{nunzia.gavitone@unina.it}
\keywords{Eigenvalue problems, nonlinear elliptic equations, Faber-Krahn inequality,Wulff shape}
\subjclass[2010]{
35P15,35P30,35J60
}
\date{\today}
\maketitle
\begin{abstract}
In this paper we study the main properties of the first eigenvalue $\lambda_{1}(\Omega)$ and its eigenfunctions of a class of highly nonlinear elliptic operators in a bounded Lipschitz domain $\Omega\subset \R^{n}$, assuming a Robin boundary condition. Moreover, we prove a Faber-Krahn inequality for $\lambda_{1}(\Omega)$.
\end{abstract}
\section{Introduction}
Let $\Omega$ be a bounded Lipschitz domain in $\R^n$, $n\ge 2$. This paper is devoted to the study of the following problem:
\begin{equation}
  \label{eq:1}
  \lambda_1(\Omega) = \min_{\substack{u\in W^{1,p}(\Omega)\\ u\neq 0}}
J(u),
\end{equation}
where
\begin{equation}\label{j}
 J(u)= \dfrac{\ds \int_\Omega [H(D u)]^{p}dx +\beta \ds \int_{\de
     \Omega}|u|^pH(\nu)d\sigma}{\ds \int_\Omega |u|^p dx},
\end{equation}
 $1<p<+\infty$, $\nu$ is the outer normal to $\de \Omega$, and $\beta$ is a fixed positive number. Moreover, we suppose that $H$ is a sufficiently smooth norm of $\R^{n}$ (see Sections 2 and 3 for the precise assumptions). The minimizers of \eqref{eq:1} satisfy the equation
 \begin{equation}
  \label{eq:2intro}
      -\divergenza\left([H(D u)]^{p-1}H_{\xi}(Du)\right)= 
      \lambda_1(\Omega) |u|^{p-2}u
      \quad\text{in } 
      \Omega,
   \end{equation}	
with Robin conditions on the boundary: 
   \begin{equation}
   \label{eq:2introbd}
      [H(D u)]^{p-1}H_{\xi}(Du)\cdot \nu +\beta H(\nu) |u|^{p-2}u=0
      \quad \text{on } \de\Omega. 
\end{equation}
The operator in \eqref{eq:2intro} reduces to the $p$-Laplacian when $H$ is the Euclidean norm of $\R^{n}$. For a general norm $H$, it is an anisotropic, highly nonlinear operator, and it has attracted an increasing interest in last years. We refer, for example, to \cite{dpg3,ferkaw,aflt} ($p=2$) and \cite{dpgtors,dpg4,bkj06,bfk} ($1<p<+\infty$) where Dirichlet boundary conditions are considered. Moreover, for Neumann boundary values see, for instance, \cite{dpg2,wxpac} ($p=2$), while overdetermined problems
are studied in \cite{ciasal,wxarma} ($p=2$). In this paper we are interested in considering the eigenvalue problem \eqref{eq:2intro} with the Robin boundary conditions \eqref{eq:2introbd}. In particular, our main objective is to obtain a Faber-Krahn inequality by studying the shape optimization problem
\begin{equation}
\label{minintro}
\min_{|\Omega|=m} \lambda_{1}(\Omega)
\end{equation}
among all the Lipschitz domains with given measure $m>0$. To study problem \eqref{minintro}, we first have to investigate the basic properties of the first eigenvalue and of the relative eigenfunctions of \eqref{eq:2intro},\eqref{eq:2introbd}, as existence, sign, simplicity and regularity.

In the Euclidean case, problem \eqref{eq:1} reduces to
\[
\lambda_{1,\mathcal E}(\Omega)= \min_{\substack{u\in W^{1,p}(\Omega)\\ u\neq 0}} 
\dfrac{\ds \int_\Omega |D u|^p dx +\beta \ds \int_{\de
     \Omega}|u|^pd\sigma}{\ds \int_\Omega |u|^p dx},
\]
and the minimizers satisfy the problem
 \begin{equation*}
  \left\{
    \begin{array}{ll}
      -\divergenza\left(|Du|^{p-2}Du\right)=\lambda_{1,\mathcal E}(\Omega) |u|^{p-2}u
      &\text{in } 
      \Omega,\\[.3cm] 
      |Du|^{p-2}\dfrac{\de u}{\de \nu} +\beta |u|^{p-2}u=0
      &\text{on } \de\Omega. 
    \end{array}
  \right.
\end{equation*}
In such a case, problem \eqref{minintro} has been first investigated by Bossel for $p=2$, when $\Omega$ varies among smooth domains of $\R^{2}$ with fixed measure. More precisely, in \cite{bos} she proved that
\begin{equation}
	\label{fkintro}
	\lambda_{1,\mathcal E}(\Omega) \ge \lambda_{1,\mathcal E}(B),
\end{equation}
where $B$ is a disk such that $|B|=|\Omega|$. This result has been generalized to any dimension $n\ge 2$ for Lipschitz domains in \cite{da06}. As regards the case $1<p<+\infty$, the inequality \eqref{fkintro} has been proved by \cite{daifu} for smooth domains, and by \cite{bd10} in the case of Lipschitz domains. The equality cases are also addressed in \cite{daifu,bd10}. As regards the case $\beta<0$, we refer the reader to \cite{fntstek} and the references therein.

In the anisotropic case, our result reads as follows. Let $H^{o}$ be the polar function of $H$, and denote by $\mathcal W_{R}$ the Wulff shape, that is the $R$-sublevel set of $H^{o}$, such that $|\mathcal W_{R}|=|\Omega|$ (see Section \ref{notation} for the definitions). If $\Omega\not=\mathcal W_{R}$ is a Lipschitz set of $\R^{n}$, then
\[
\lambda_{1}(\Omega)> \lambda_{1}(\mathcal W_{R}).
\] 
Hence, the unique minimizer of \eqref{minintro} is the Wulff shape. Such result relies in the so-called anisotropic isoperimetric inequality (see for example \cite{aflt}), and it is in agreement with the Faber-Krahn inequality for the first eigenvalue of \eqref{eq:2intro} in the homogeneous Dirichlet case (see \cite{bfk}).

As a matter of fact, we may ask if the first eigenvalue $\lambda_{1}(\Omega)$ is bounded from above in terms of the Lebesgue measure of $\Omega$. Indeed, in the Euclidean setting, this is the case for the first nonvanishing Neumann Laplacian eigenvalue (see \cite{w56}, and also \cite{chdb12,bct13} for related results), but this does not happen for the first Dirichlet Laplacian eigenvalue. In this order of ideas, by a result given in \cite{kov} it follows that the first Robin Laplacian eigenvalue among the sets of fixed measure is unbounded from above.
Here we prove a lower bound for the first eigenvalue $\lambda_{1}(\Omega)$ of our anisotropic Robin problem in a convex set $\lambda_{1}(\Omega)$ in terms of the anisotropic inradius of $\Omega$. This will imply that, among all Lipschitz sets with fixed measure $m>0$,
\[
\sup_{|\Omega|=m}\lambda_{1}(\Omega)=+\infty.
\]

The paper is organized as follows. In Section 2, we recall some basic definitions and properties of $H$ and of its polar function $H^{o}$. In Section 3, we state and prove some properties of the first eigenvalue of \eqref{eq:2intro}, \eqref{eq:2introbd}. More precisely, under suitable assumptions on $H$, we show that there exists a first eigenvalue $\lambda_{1}(\Omega)$ which is simple. Moreover, we prove that the first eigenfunctions are in $C^{1,\alpha}(\Omega)\cap C(\bar\Omega)$, for some $0<\alpha<1$. Furthermore, a solution of the eigenvalue problem is a first eigenfunction if and only if it has a fixed sign. In Section 4 we investigate the eigenvalue problem when $\Omega$ is a Wulff shape, while in Section 5 we give a representation formula for $\lambda_{1}(\Omega)$ by means of the level sets of the first eigenfunctions. Using such results, in Section 6 we state precisely the main result and give a proof.

\section{Notation and preliminaries}
\label{notation}
Let $H:\R^n\rightarrow [0,+\infty[$, $n\ge 2$, be a $C^2(\mathbb R^n\setminus\{0\})$ function
such that 
\begin{equation}\label{eq:omo}
  H(t\xi)= |t| H(\xi), \quad \forall \xi \in \R^n,\; \forall t \in
  \R,
\end{equation}
and such that any level set $\{\xi\in \R^{n}\colon H(\xi)\le t\}$, 
with $t>0$ is strictly convex. 

Moreover, suppose that there exist two positive constants $a
\le b$ such that
\begin{equation}\label{eq:lin}
  a|\xi| \le H(\xi) \le b|\xi|,\quad \forall \xi\in \R^n.
\end{equation}
 
\begin{rem}
	We stress that the homogeneity of $H$ and the convexity of its level sets imply the convexity of $H$. Indeed, by \eqref{eq:omo}, it is sufficient to show that, for any $\xi_{1},\xi_{2}\in \R^{n}\setminus\{0\}$,
	\begin{equation}
	\label{tesirem}
		H(\xi_{1}+\xi_{2}) \le H(\xi_{1}) + H(\xi_{2}).
	\end{equation}
	By the convexity of the level sets, we have
	\begin{multline*}
		H\!\left(\frac{\xi_{1}}{H(\xi_{1})  + H(\xi_{2})} + 
	\frac{\xi_{2}}{H(\xi_{1}) \! +\! H(\xi_{2})}
	\right)=\\ =
		H\left(\frac{H(\xi_{1})}{H(\xi_{1}) + H(\xi_{2})} \frac{\xi_{1}}{H(\xi_{1})} + 
	\frac{H(\xi_{2})}{H(\xi_{1})+ H(\xi_{2})} \frac{\xi_{2}}{H(\xi_{2})}
	\right) \le 1,
	\end{multline*}
	and by \eqref{eq:omo} we get \eqref{tesirem}.
\end{rem}

We define the polar function $H^o\colon\R^n \rightarrow [0,+\infty[$ 
of $H$ as
\begin{equation*}
H^o(v)=\sup_{\xi \ne 0} \frac{\xi\cdot v}{H(\xi)}. 
\end{equation*}
 It is easy to verify that also $H^o$ is a convex function
which satisfies properties \eqref{eq:omo} and
\eqref{eq:lin}. Furthermore, 
\begin{equation}
\label{hh0def}
H(v)=\sup_{\xi \ne 0} \frac{\xi\cdot v}{H^o(\xi)}.
\end{equation}
The set
\[
\mathcal W = \{  \xi \in \R^n \colon H^o(\xi)< 1 \}
\]
is the so-called Wulff shape centered at the origin. We put
$\kappa_n=|\mathcal W|$, where $|\mathcal W|$ denotes the Lebesgue measure
of $\mathcal W$. More generally, we denote with $\mathcal W_r(x_0)$
the set $r\mathcal W+x_0$, that is the Wulff shape centered at $x_0$
with measure $\kappa_nr^n$, and $\mathcal W_r(0)=\mathcal W_r$.

The following properties of $H$ and $H^o$ hold true
(see for example \cite{bp}):
\begin{gather}
 H_{\xi}(\xi) \cdot \xi = H(\xi), \quad  H_{\xi}^{o} (\xi) \cdot \xi 
= H^{o}(\xi),\label{eq:om}
 \\
 H( H_{\xi}^o(\xi))=H^o( H_{\xi}(\xi))=1,\quad \forall \xi \in
\R^n\setminus \{0\}, \label{eq:H1} \\
H^o(\xi)  H_{\xi}( H_{\xi}^o(\xi) ) = H(\xi) 
H_{\xi}^o( H_{\xi}(\xi) ) = \xi,\quad \forall \xi \in
\R^n\setminus \{0\}. \label{eq:HH0}
\end{gather}
\begin{definiz}[Anisotropic area functional and perimeter (\cite{bu,and})] 
\label{perdef}
Let $M$ be an oriented $(n-1)$-dimensional hypersurface in $\mathbb R^{n}$.
The anisotropic area functional of $M$ is 
\[
	\sigma_{H}(M):=\int_{M} H(\nu)\,d \sigma,
\]
where $\nu$ denotes the outer normal to $M$ and $\sigma$ is the $(n-1)$-dimensional Hausdorff measure.
\end{definiz}

The anisotropic area of a set $M$ is
finite if and only if the usual Euclidean hypersurface area $\sigma(M)$ is
finite. Indeed, by property \eqref{eq:lin} we
have that
\begin{equation*}
	\alpha \,\sigma(M) \le \sigma_H(M) \le \gamma\, \sigma(M).
\end{equation*}

An isoperimetric inequality for the anisotropic area holds, namely for $K\subset M$ open set of $\mathbb R^{n}$ with Lipschitz boundary,
\begin{equation}
  \label{isop}
  \sigma_H(\de K) \ge n\kappa_n^{\frac 1 n} |K|^{1-\frac 1 n},
\end{equation}
and the equality holds if and only if $K$ is homothetic to a Wulff shape (see for example \cite{bu}, \cite{dpf}, \cite{fomu}, \cite{aflt}). We stress that in \cite{dpg1} an isoperimetric inequality for the anisotropic relative
perimeter in the plane is studied.

Let $\Omega$ be a bounded  open set of $\R^n$, and $d_H(x)$ the
anisotropic distance of a point $x\in \Omega$ to the boundary $\de
\Omega$ , that is
\begin{equation}
\label{defdist}
d_H(x)= \inf_{y\in \de \Omega} H^o(x-y).
\end{equation}
By the property \eqref{eq:H1}, the distance function $d_H(x)$
satisfies
\begin{equation}
  \label{Hd}
  H(D d_H(x))=1.
\end{equation}
Finally, we recall that when $\Omega$ is convex $d_H(x)$
is concave. In a natural way, the anisotropic inradius of a convex, bounded open set $\Omega$ is the value
\begin{equation}
\label{inrad}
R_{H,\Omega}=\sup \{d_{H}(x),\; x\in\Omega\}
\end{equation}
For further properties of the anisotropic distance function we refer
the reader to \cite{cm07}.

\section{The first eigenvalue problem}
In this section we prove some properties of the minimizers of \eqref{eq:1}, which are the weak solutions of the following Robin boundary value problem: 
\begin{equation}
  \label{eq:2}
  \left\{
    \begin{array}{ll}
      -\divergenza\left(F_p(Du)\right)=\lambda_1(\Omega) |u|^{p-2}u
      &\text{in } 
      \Omega,\\[.2cm] 
      F_p(Du)\cdot \nu +\beta H(\nu) |u|^{p-2}u=0
      &\text{on } \de\Omega. 
    \end{array}
  \right.
\end{equation}
where
\begin{equation*}
     F_p (Du):= [H(D u)]^{p-1}H_{\xi}(Du).
\end{equation*}
For weak solution of problem \eqref{eq:2} we mean a function $u\in W^{1,p}(\Omega)$ such that
  \begin{equation}
  \label{defsol}
  \int_{\Omega} F_{p}(Du)\cdot D\psi\, dx +\beta \int_{\de \Omega} u^{p-1}\psi\, H(\nu)\,d\sigma =\lambda_{1}(\Omega)\int_{\Omega} |u|^{p-2}u\,\psi\,dx,\qquad \psi\in W^{1,p}(\Omega).
  \end{equation}
  
Obviously, $\lambda_{1}(\Omega)$ in \eqref{eq:1} (and then in \eqref{eq:2}) depends also on $\beta$. In general, we will consider $\beta>0$ fixed. Anyway, when it will be necessary, to emphasize the dependence on $\beta$ we will denote the first eigenvalue of \eqref{eq:2} with $\lambda_{1}(\Omega,\beta)$.

For the Euclidean case we refer to \cite{le06}, where the eigenvalue problem for the $p$-Laplacian under several boundary conditions is considered.

From now on, we assume that $H$ is a convex function as in Section \ref{notation}, assuming also that it verifies the following hypothesis: 
\begin{equation}
  \label{ipellipt}
H\in C^2(\Rn\setminus\{0\}),\text{ with
}\sum_{i,j=1}^{n}{\dfrac{\de}{\de \xi_j}
  \left( [H(\eta)]^{p-1}H_{\xi_i}(\eta)\right)\xi_i\xi_j}\ge
\gamma |\eta|^{p-2} |\xi|^2, 
\end{equation}
for some positive constant $\gamma$, for any $\eta \in
\Rn\setminus\{0\}$ and for any $\xi\in \Rn$. 
\begin{theo}
\label{teoremone}
There exists a function $u_p \in C^{1,\alpha}(\Omega)\cap C(\bar\Omega)$ which realizes the minimum in \eqref{eq:1}, and satisfies the problem \eqref{eq:2}. Moreover, $\lambda_1(\Omega)$ is the first eigenvalue of \eqref{eq:2}, and the first eigenfuctions are positive (or negative) in $\Omega$. 
\end{theo}

\begin{proof}
The proof makes use of standard arguments. We briefly recall the main steps.
The direct method of the Calculus of Variations guarantees that the infimum in \eqref{eq:1} is attained at a function $u_p \in W^{1,p}(\Omega)$. We may assume that $u_p\ge 0$, being also $|u_p|$ a minimizer in \eqref{eq:1}. Moreover, the function $u_{p}$ is a weak solution of \eqref{eq:2}.
    In order to obtain that $u_{p}\in C^{1,\alpha}(\Omega)\cap C(\bar\Omega)$, we first claim that a $L^{\infty}$-estimate for $u_{p}$ holds. To get the claim, we take $\varphi= [T_{M}(u_{p})]^{kp+1}$ as test function, with $k,M$  positive numbers, and $T_{M}(s)=\min\{s,M\}$, $s\ge 0$. Using \eqref{eq:om} and \eqref{eq:lin}, we easily get
  \begin{multline*}
	\alpha(kp+1) \int_{u_{p}\le M}|Du_{p}|^{p}u_{p}^{kp}\,dx \le \\ \le  
\int_{\Omega} F_{p}(Du_{p})\cdot D\varphi\, dx +\beta \int_{\de \Omega} u_{p}^{p-1}\varphi\, H(\nu)\,d\sigma \le \\ \le
\lambda_{1}(\Omega)\int_{\Omega} u_{p}^{p(k+1)}\,dx, 
\end{multline*}
and then
\begin{equation*}
\int_{\Omega}\big|D\,T_{M}(u_{p})^{k+1}\big|^{p}\,dx +\int_{\Omega} [T_{M}(u_{p})]^{p(k+1)}\,dx \le \left( \dfrac{(k+1)^p}{\alpha(kp+1)} \lambda_{1}(\Omega)+1\right)\int_{\Omega} u_{p}^{p(k+1)}\,dx.
\end{equation*}

Applying the Sobolev inequality and the Fatou lemma, we get that
\begin{equation*}
	\|u_{p}\|_{(k+1)p^{*}} \le S^{\frac{1}{k+1}}\left( \dfrac{(k+1)^p}{kp+1} 	\frac{\lambda_{1}(\Omega)}{\alpha}+1\right)^{\frac{1}{p(k+1)}} 
	\|u_{p}\|_{(k+1)p}, 
\end{equation*}
where $S$ is the Sobolev constant. Using the standard Moser iteration technique for the $L^{p}$-norms, we get the claim. For sake of completeness, we give the complete proof (see also \cite{gt}). 

First of all, we have that there exists a constant $c$ independent of $k$ such that
\[
\left( \dfrac{(k+1)^p}{kp+1} 	\frac{\lambda_{1}(\Omega)}{\alpha}+1\right)^{\frac{1}{p\sqrt{k+1}}} \le c.
\]
Then,
\begin{equation}
\label{moser1}
	\|u_{p}\|_{(k+1)p^{*}} \le S^{\frac{1}{k+1}} c ^{\frac{1}{\sqrt{k+1}}} 
		\|u_{p}\|_{(k+1)p}.
\end{equation}
Choosing $k_{n}$ in \eqref{moser1} such that $(k_{1}+1)p=p^{*}$, and $k_{n}$, $n\ge 2$, such that $(k_{n}+1)p=(k_{n-1}+1)p^{*}$, by induction we obtain
\[
	\|u_{p}\|_{(k_{n}+1)p^{*}} \le S^{\frac{1}{k_{n}+1}} c ^{\frac{1}{\sqrt{k_{n}+1}}} 
		\|u_{p}\|_{(k_{n-1}+1)p^{*}}.
\]
Hence, using iteratively the above inequality, we get
\[
	\|u_{p}\|_{(k_{n}+1)p^{*}} \le S^{\sum_{i=1}^{n}\frac{1}{k_{i}+1}} c ^{\sum_{i=1}^{n}\frac{1}{\sqrt{k_{i}+1}}} 
		\|u_{p}\|_{p^{*}}.
\]
Being $k_{n}+1=(p^{*}/p)^{n}$, and $p^{*}/p>1$, it follows that for any $n\ge 1$
\begin{equation}
\label{moser2}
	\|u_{p}\|_{(k_{n}+1)p^{*}} \le C \|u\|_{p^{*}},
\end{equation}
as $r_{n}=(k_{n}+1)p^{*}\rightarrow +\infty$ as $n\rightarrow +\infty$. The estimates in \eqref{moser2} imply that $u\in L^{\infty}(\Omega)$. Indeed, if by contradiction the exist $\eps>0$ and $A\subset \Omega$ with positive measure such that $|u|>C\|u\|_{p^{*}}+\eps=K$ in $A$, we have
\[
\liminf_{n} \|u\|_{r_{n}} \ge \liminf_{n} \left(\int_{A}K^{r_{n}}\right)^{\frac{1}{r_{n}}}= K>C\|u\|_{p^{*}},
\]
which is in contrast with \eqref{moser2}. 

Now the $L^{\infty}$-estimate, the hypothesis \eqref{ipellipt} and the properties of $H$ allow to apply standard regularity results (see \cite{db83}, \cite{tk84}), in order to obtain that $u\in C^{1,\alpha}(\Omega)$. As matter of fact, as observed in \cite{bd10} it is possible to follow the argument in \cite[pages 466-467]{ladyz} to get the continuity of $u_{p}$ up to the boundary. Finally, $u_{p}$ is strictly positive in $\Omega$ by the Harnack inequality (see \cite{truharn}). 
\end{proof}

\begin{theo}
	The first eigenvalue $\lambda_{1}(\Omega)$ of \eqref{eq:2} is simple, that is the relative eigenfunctions are unique up to a multiplicative constant.
\end{theo}
\begin{proof}
We follow the idea of \cite{bfk,bk02}. Let $v,w$ two positive minimizers of \eqref{eq:1} in $\Omega$ such that $\|v\|_{p}=\|w\|_{p}=1$, and consider $\eta_{t}=(t v^{p}+(1-t)w^{p})^{1/p}$, with $t\in[0,1]$.  Obviously, $\|\eta_{t}\|_{p}=1$. Moreover, using the homogeneity and the convexity of $H$ we get that
\begin{equation}
	\label{convx}
	\begin{array}{rl}
		[H(D\eta_{t})]^{p} & = \eta_{t}^{p} \left[H\left( t 
		\left(\dfrac{v}{\eta_{t}} \right)^{p}
		\dfrac{Dv}{v} + (1-t)\left(\dfrac{w}{\eta_{t}} \right)^{p}		\dfrac{Dw}{w} \right)\right]^{p} \\[.4cm]
		&=\eta_{t}^{p} \left[H\left( s(x) 
		\dfrac{Dv}{v} + (1-s(x)) \dfrac{Dw}{w} \right)\right]^{p} \\[.4cm]
		&\le \eta_{t}^{p} \left[ s(x) H\left(
		\dfrac{Dv}{v}\right) + (1-s(x)) H \left( \dfrac{Dw}{w}\right) \right]^{p} \\[.4cm]
		&\le tv^{p} \left[H\left(
		\dfrac{Dv}{v}\right)\right]^{p} + (1-t)w^{p} \left[H \left( \dfrac{Dw}{w}\right)\right]^{p} \\[.4cm]
		 &= t [H(D v)]^{p}+(1-t)[H(Dw)]^{p}.
	\end{array}
\end{equation}
Hence, recalling \eqref{j}, the inequalities in \eqref{convx} and the definition of $\eta_{t}$ give that
\[
J(\eta_{t}) \le t J(v)+ (1-t)J(w) =\lambda_{1}(\Omega),
\]
and then $\eta_{t}$ is a minimizer for $J$. This implies that the inequalities in \eqref{convx} become equalities. The equality between the third and the fourth row of \eqref{convx} holds if and only if $H(Dv/v)=H(Dw/w)$. Hence, the strict convexity of the level sets of $H$ guarantees from the equalities in \eqref{convx} that $D v/v= Dw/w$ in $\Omega$, that is $v/w$ is constant. The norm constraint on $v$ and $w$ implies the uniqueness, and this concludes the proof.  
\end{proof}

\begin{rem}
\label{remteoremone}
We stress that the nonnegative solution $u_{p}\in C^{1,\alpha}(\Omega)\,\cap\, C(\bar \Omega)$ of \eqref{eq:2} we found by Theorem \ref{teoremone} cannot be identically zero on $\de \Omega$. Indeed, in such a case, taking $\psi=1$ as test function in \eqref{defsol}, we obtain 
\[
\int_{\Omega} u_{p}^{p-1}dx=0, 
\]
contradicting the positivity of $u_{p}$ in $\Omega$. As a matter of fact, if we suppose $\de \Omega$ to be a connected $C^{2}$ manifold, then the Hopf boundary point Lemma holds (see \cite{ct00}), which implies that $u$ cannot vanish on $\de \Omega$. 
\end{rem}
\begin{theo}
\label{theodue}
	Any nonnegative function $v\in W^{1,p}(\Omega)$, $v\not\equiv 0$, which satisfies, in the sense of \eqref{defsol},
	\begin{equation}
	\label{pbdue}
 	\left\{
   	\begin{array}{ll}
      	-\divergenza\left(F_p(Dv)\right)=\lambda v^{p-1}
      	&\text{in } \Omega,\\[.2cm] 
      	F_p(Dv)\cdot \nu +\beta H(\nu)\, v^{p-1}=0
      	&\text{on } \de\Omega. 
    	\end{array}
  \right.
	\end{equation}
is a first eigenfunction of \eqref{pbdue}, that is $\lambda=\lambda_{1}(\Omega)$ and $v=u_{p}$, where $u_{p}$ is given in Theorem \ref{teoremone}, up to multiplicative constant.
\end{theo}
For analogous results in the Dirichlet case, see for example \cite{klp} and the references therein.
\begin{proof}[Proof of Theorem \ref{theodue}]
The same arguments of Theorem \ref{teoremone} allow to prove that the given nonnegative solution $v$ of \eqref{pbdue} is in $C^{1,\alpha}(\Omega)\cap C(\bar\Omega)$ and it is positive in $\Omega$. Moreover, the function $u_{p}\in C^{1,\alpha}(\Omega)\cap C(\bar\Omega)$ satisfies
	\begin{equation}
	\label{1}
		\int_{\Omega} [H(Du_{p})]^{p} dx +\beta \int_{\de\Omega} u_{p}^{p} H(\nu)\,d\sigma =\lambda_{1}(\Omega) \int_{\Omega} u_{p}^{p}\,dx,
	\end{equation}
	while, choosing $u_{p}^{p}/(v+\eps)^{p-1}$, with $\eps>0$, as test function for $v$, we get
	\begin{multline}
	\label{2}
		\int_{\Omega}
		p\, \left[H\Big(  \frac{u_{p}}{v+\eps} Dv \Big)\right]^{p-1} H_{\xi}(Dv)\cdot D u_{p} \, dx - (p-1) \int_{\Omega} \left[H\Big( \frac{u_{p}}{v+\eps}  Dv\Big)\right]^{p} \, dx
		+\\+ \beta \int_{\de\Omega} \frac{v^{p-1}}{(v+\eps)^{{p-1}}} \, u_{p}^{p}\, H(\nu)\,d\sigma = \lambda \int_{\Omega}  \frac{v^{p-1}}{(v+\eps)^{{p-1}}} \, u_{p}^{p} \,dx.
	\end{multline}
	Subtracting \eqref{2} by \eqref{1}, being $H_{\xi}$ zero homogeneous, and observing that $v/(v+\eps)\le 1$, we get
	\begin{multline*}
	\int_{\Omega} \left\{[ H(Du_{p})]^{p} - F_{p}\Big(  \frac{u_{p}}{v+\eps} Dv \Big) \cdot D u_{p}  + (p-1) \left[H\Big( \frac{u_{p}}{v+\eps}  Dv\Big)\right]^{p} 
	 \right\}dx \le \\ \le
	 \int_{\Omega} \left[\lambda_{1}(\Omega)-\frac{v^{p-1}}{(v+\eps)^{{p-1}}}\lambda\right] u_{p}^{p}\, dx.
	\end{multline*}
	The convexity of $H^{p}$ guarantees that the left-hand side in the above inequality is nonnegative. Hence, as $\eps \rightarrow 0$, the monotone convergence gives that
	\[
	(\lambda_{1}(\Omega)-\lambda) \int_{\Omega} u_{p}^{p}\, dx \ge 0,
	\]
	and this can hold if and only if $\lambda \le\lambda_{1}(\Omega)$. Being $\lambda_{1}(\Omega)$ the smallest possible eigenvalue, necessarily we have that $\lambda=\lambda_{1}(\Omega)$. The uniqueness of the first eigenfuction implies that, up to some positive multiplicative constant, $v=u_{p}$.
\end{proof}
In order to show a lower bound for $\lambda_{1}(\Omega)$ when $\Omega$ is a convex set of $\R^{n}$ in terms of the anisotropic inradius of $\Omega$, we need an Hardy-type inequality for functions which, in general, do not vanish on the boundary. To this aim, we impose further regularity on $H$. More precisely, we assume also that
\begin{equation}
\label{ipcm}
\de W=\{x\colon H^{o}(x)=1\}\text{ has positive Gaussian curvature in any point.}
\end{equation}
If $\Omega$ is $C^{2}$, this assumption ensures that the anisotropic distance from the boundary of $\Omega$ is $C^{2}$ in a tubular neighborhood of $\de \Omega$ (see for instance \cite{cm07}). 
\begin{lemma}
Let $\Omega$ be a bounded convex open set of $\R^{n}$ with $C^2$ boundary and suppose that $H^{o}$ satisfies also \eqref{ipcm}. Then, for any $\alpha>0$ and $\vartheta>0$, the following Hardy-type inequality holds:
\begin{equation}
\label{lemmahardy}
\int_{\Omega} [H(Du)]^{p} dx + \vartheta^{p-1} 
\int_{\de\Omega} |u|^{p}H(\nu)d\sigma 
\ge (p-1)(\alpha\vartheta)^{p-1}(1-\alpha\vartheta) \int_{\Omega} \frac{|u|^{p}}{(d_{H}+\alpha)^{p}} dx,
\end{equation}
where $u\in W^{1,p}(\Omega)$ and $d_{H}$ is the anisotropic distance from the boundary of $\Omega$, defined in \eqref{defdist}.
\end{lemma}
\begin{proof}
It sufficient to prove the thesis for $u\ge 0$. Moreover, using an approximation argument, we can suppose that $u\in C^{1}(\bar\Omega)$. For $\delta$ positive, let us define $H_{\delta}(\xi)=H^{\delta}(\xi)+\delta$, where $H^{\delta}$ is the $\delta$-mollification of $H$. 
By the convexity of $H(\xi)$, the function $H^{\delta}$ is convex and we have, for any $\xi_{1},\xi_{2}\in\R^{n}$, 
\[
[H_{\delta}(\xi_{1})]^{p} \ge [H_{\delta}(\xi_{2})]^{p} + p [H_{\delta}(\xi_{2})]^{p-1}(H_{\delta})_{\xi}(\xi_{2})\cdot(\xi_{1}-\xi_{2}).
\]
We apply the above inequality to $\xi_{1}=Du$ and  $\xi_{2}=\dfrac{\alpha\vartheta u}{d^{\eps}+\alpha}D d^{\epsilon}$, where $\alpha>0$, $\vartheta>0$, and $d^{\epsilon}$ is the $\epsilon$-mollification of $d_H$. The convexity of $\Omega$ gives that the function $d_H$, and then $d^{\epsilon}$, are  concave functions. We have:
\begin{multline}
\label{mul:hardy}
\int_{\Omega} [H_{\delta}(Du)]^{p} dx \ge
(\alpha\vartheta)^{p}\int_{\Omega}\frac{u^{p}}{(d^{\epsilon}+\alpha)^{p}}[H_{\delta}(Dd^{\epsilon})]^{p} dx +\\+ p (\alpha\vartheta)^{p-1}\int_{\Omega} \frac{u^{p-1}}{(d^{\epsilon}+\alpha)^{p-1}} [H_{\delta}(Dd^{\epsilon})]^{p-1}(H_{\delta})_{\xi}(Dd^{\epsilon})\cdot Du \,dx +\\- 
p(\alpha\vartheta)^{p}\int_{\Omega} \frac{u^{p}}{(d^{\epsilon}+\alpha)^{p}}[H_{\delta}(Dd^{\epsilon})]^{p-1} (H_{\delta})_{\xi}(Dd^{\epsilon})\cdot Dd^{\epsilon} dx
\end{multline}
Passing to the limit as $\delta\rightarrow 0$ and using \eqref{eq:om}, the sum of the first and the third terms in the right-hand side of \eqref{mul:hardy} converge to
\[
-(p-1)(\alpha\vartheta)^{p}\int_{\Omega}\frac{u^{p}}{(d^{\epsilon}+\alpha)^{p}} [H(Dd^{\epsilon})]^{p}dx.
\]
Moreover, by the divergence theorem 
we have that 
\begin{multline}
\label{mul:hardy2}
p \int_{\Omega} \frac{u^{p-1}}{(d^{\epsilon}+\alpha)^{p-1}}[H_{\delta}(Dd^{\epsilon})]^{p-1} (H_{\delta})_{\xi}(Dd^{\epsilon})\cdot Du\,dx=
\\ 
=\frac 1 p \int_{\Omega} \frac{1}{(d^{\epsilon}+\alpha)^{p-1}} (H_{\delta}^{p})_{\xi}(Dd^{\epsilon})\cdot D(u^{p})\,dx=
\\
=\frac 1 p \int_{\de\Omega} \frac{u^{p}}{(d^{\epsilon}+\alpha)^{p-1}}(H_{\delta}^p)_{\xi}\left(D d^{\epsilon} \right)\cdot \nu\, d\sigma
-\frac 1 p \int_{\Omega} u^{p} 
\divergenza \left(
\frac{(H_{\delta}^p)_{\xi}(D d^{\epsilon})}{(d^{\epsilon}+\alpha)^{p-1}}\right)dx = 
\\ =
\frac 1 p \int_{\de\Omega} \frac{u^{p}}{(d^{\epsilon}+\alpha)^{p-1}}(H_{\delta}^p)_{\xi}\left(D d^{\epsilon} \right)\cdot \nu\, d\sigma - \frac 1 p  \int_{\Omega} \frac{u^{p}}{(d_{H}+\alpha)^{p-1}}\divergenza((H_{\delta}^p)_\xi(D d^{\epsilon})) dx 
+\\
+\frac{p-1}{p} \int_{\Omega} \frac{u^{p}}{(d^{\epsilon}+\alpha)^{p}} (H_{\delta}^p)_\xi(D d^{\epsilon})\cdot Dd^{\epsilon}
 dx 
 \ge \\
 \ge \frac 1 p\int_{\de\Omega} \frac{u^{p}}{(d^{\epsilon}+\alpha)^{p-1}}(H_{\delta}^p)_{\xi}\left(D d^{\epsilon} \right)\cdot \nu\, d\sigma 
 + \frac{p-1}{p} \int_{\Omega} \frac{u^{p}}{(d^{\epsilon}+\alpha)^{p}} (H_{\delta}^p)_\xi(D d^{\epsilon})\cdot Dd^{\epsilon}
 dx 
\end{multline}
Last inequality follows from the fact that $-\divergenza((H_{\delta}^p)_\xi (D d^{\epsilon}))$ is nonnegative. Indeed, it is the trace of the product of the matrices
$\big[(H_{\delta}^{p})_{\xi\xi}(D d^{\epsilon})\big]$ and $\big[-D^{2}d^{\epsilon}\big]$, which are both positive semidefinite, being $H_{\delta}^{p}$ convex and $d^{\epsilon}$ concave.

Passing to the limit as $\delta\rightarrow 0$ in \eqref{mul:hardy2}, and using \eqref{eq:om}, we get 
\begin{multline*}
p \int_{\Omega} \frac{u^{p-1}}{(d^{\epsilon}+\alpha)^{p-1}}[H(Dd^{\epsilon})]^{p-1} (H)_{\xi}(Dd^{\epsilon})\cdot Du\,dx \ge \\
\ge \frac 1 p\int_{\de\Omega} \frac{u^{p}}{(d^{\epsilon}+\alpha)^{p-1}}(H^p)_{\xi}\left(D d^{\epsilon} \right)\cdot \nu\, d\sigma 
 +(p-1) \int_{\Omega} \frac{u^{p}}{(d^{\epsilon}+\alpha)^{p}} [H (Dd^{\epsilon})]^{p}
 dx 
\end{multline*}
Then, as $\delta\rightarrow 0$ in \eqref{mul:hardy}, the above computations gives that
\begin{multline}
\label{ultappr}
\int_{\Omega} [H(Du)]^{p} dx - \frac{(\alpha\vartheta)^{p-1}}{p} 
\int_{\de\Omega} u^{p}(H^p)_\xi \left( \frac{D d^{\epsilon}}{d^{\epsilon}+\alpha}\right)\cdot \nu\, d\sigma \ge \\
\ge (p-1)(\alpha\vartheta)^{p-1}(1-\alpha\vartheta) \int_{\Omega} \frac{u^{p}}{(d^{\epsilon}+\alpha)^{p}} [H(Dd^{\epsilon})]^{p} dx.
\end{multline}

Now we pass to the limit for $\epsilon \to 0$. Recalling that under our assumptions $d_H$ is $C^2$ in a tubular neighborhood of $\de \Omega$, by uniform convergence we get
\begin{multline}
\int_{\Omega} [H(Du)]^{p} dx-(\alpha\vartheta)^{p-1}\int_{\de\Omega} \frac{u^{p}}{(d_H+\alpha)^{p-1}}[H(Dd_H)]^{p-1} H_{\xi}(Dd_H) \cdot \nu d\sigma \ge \\ \ge (p-1)(\alpha\vartheta)^{p-1}(1-\alpha\vartheta) \int_{\Omega} \frac{u^{p}}{(d_{H}+\alpha)^{p}} [H(Dd_H)]^{p} dx.
\end{multline}
Being $H(Dd_H)=1$ a.e. in $\Omega$, and $d_H=0 $ on $\de \Omega$, choosing $\vartheta^{p-1}=\beta$, and recalling that $\nu=-Dd_{H}/|Dd_{H}|$ on $\de \Omega$, by \eqref{eq:omo} and \eqref{eq:om} we get the thesis.
\end{proof}
An immediate application of the previous Lemma is the following result.
\begin{prop}
If $\Omega$ is a convex set of $\R^{n}$ with $C^2$ boundary and if $H$ satisfies also \eqref{ipcm}, then
\[
\lambda_{1}(\Omega) \ge \left(\frac{p-1}{p}\right)^{p} 
\frac{\beta}{R_{H,\Omega}\left(1+\beta^{\frac{1}{p-1}}R_{H,\Omega}\right)^{p-1}},
\]
where $R_{H,\Omega}$ is the anisotropic inradius of $\Omega$, as defined in \eqref{inrad}.
\end{prop}
\begin{proof}
Let $\beta=\vartheta^{p-1}$. Then, by \eqref{lemmahardy} and the definitions of $\lambda_{1}(\Omega)$ and of the anisotropic inradius $R_{H,\Omega}$ we get that
\[
\lambda_{1}(\Omega) \ge \frac{(p-1)\beta}{
(R_{H,\Omega}+\alpha)^{p}} (1-\beta^\frac{1}{p-1}\alpha) \alpha^{p-1}.
\]
Then, maximizing the right-hand side of the above inequality we obtain that
\[
\lambda_{1}(\Omega) \ge \left(\frac{p-1}{p}\right)^{p} 
\frac{\beta}{R_{H,\Omega}\left(1+\beta^{\frac{1}{p-1}}R_{H,\Omega}\right)^{p-1}}.
\]
\end{proof}
\begin{rem}
As a consequence of the previous Proposition, we have that
\[
\sup_{|\Omega|=m}\lambda_{1}(\Omega)=+\infty.
\]
among all the Lipschitz domains with given measure $m>0$.
\end{rem}
Finally, we have the following scaling property. 
\begin{prop}
	For any $t>0$, we have that $\lambda_{1}(t\Omega,\beta)=t^{-p}\lambda_{1}(\Omega,t^{p-1}\beta)$.
\end{prop}
\begin{proof}
By the homogeneity of $H$, we have: 
\begin{multline*}
\lambda_{1}(t\Omega,\beta)=\min_{\substack{v\in W^{1,p}(t\Omega)\\ v\neq 0}}
\frac{\ds\int_{t\Omega} [H(Dv(x))]^{p}dx +\beta\int_{\de(t\Omega)}|v(x)|^{p}H(\nu(x))d\sigma(x)}{\ds\int_{t\Omega}|v(x)|^{p}dx}\\[.15cm]
=\min_{\substack{u\in W^{1,p}(\Omega)\\ u\neq 0}}
\frac{t^{n-p}\ds\int_{\Omega} [H(Du(y))]^{p}dy +t^{n-1}\beta\int_{\de\Omega}|u(y)|^{p}H(\nu(y))d\sigma(y)}{t^{n}\ds\int_{\Omega}|u(y)|^{p}dy}
\\=t^{-p}\lambda_{1}(\Omega,t^{p-1}\beta).
\end{multline*}
\end{proof}

\section{The eigenvalue problem in the anisotropic radial case}
In this section we study the properties of the minimizers of \eqref{eq:1} when $\Omega$ is homothetic to the Wulff shape, that is, for $R>0$, the functions $v_{p}$ such that
\begin{equation}
\label{minrad}
J(v_{p})=\min_{\substack{u\in W^{1,p}(\mathcal W_{R}) \\ u\not\equiv 0}} \dfrac{\ds\int_{\mathcal W_{R}} [H(Du)]^{p}dx +\beta\int_{\de \mathcal W_{R}}|u|^{p}H(\nu)d\sigma}{\ds\int_{\mathcal W_{R}} |u|^{p}dx},
\end{equation}
where $\mathcal W_{R}=R\,\mathcal W=\{x\colon H^{o}(x)<R\}$, with $R>0$, and $\mathcal W$ is the Wulff shape centered at the origin. By Theorem \ref{teoremone}, such functions solve the following problem:
\begin{equation}
  \label{eq:3}
  \left\{
    \begin{array}{ll}
      -\divergenza\left(F_p(Dv)\right)=\lambda_1(\mathcal W_{R}) 
      |v|^{p-2}v 
      &\text{in } 
      \mathcal W_{R},\\[.2cm] 
      F_p(Dv)\cdot \nu +\beta H(\nu) |v|^{p-2}v=0
      &\text{on } \de\mathcal W_{R}.
    \end{array}
  \right.
\end{equation}

\begin{theo}
\label{teorad}
Let $v_{p}\in C^{1,\alpha}(\Omega)\cap C(\bar \Omega)$ be a positive solution of problem \eqref{eq:3}. Then, there exists a 
 decreasing function $\varrho_{p}=\varrho_{p}(r)$, $r\in [0,R]$, such that 
$\varrho_{p}\in C^{\infty}(0,R)\cap C^{1}([0,R])$, and
\[
	\left\{
		\begin{array}{l}
			v_{p}(x)=\varrho_{p}(H^{o}(x)),\; x\in  \overline{\mathcal W}_{R},
			\\[.15cm]
			\varrho_{p}'(0)=0,\\[.15cm]
			-(-\varrho_{p}'(R))^{p-1} + \beta (\varrho_{p}(R)) ^{p-1}=0.
		\end{array} 
	\right.
\]
\end{theo}

\begin{proof}
Let $B_{R}$ be the Euclidean ball centered at the origin, $B_{R}=\{x\in \R^{n}\colon |x|<R\}$, and consider the $p$-Laplace eigenvalue problem in $B_{R}$, that is \eqref{eq:3} with $H(\xi)=|\xi|$:
\begin{equation}
  \label{eq:iso3}
  \left\{
    \begin{array}{ll}
      -\Delta_{p} w= \lambda_{1,\mathcal E}(B_{R}) 
      |w|^{p-2}w 
      &\text{in } 
       B_{R},\\[.2cm] 
      |Dw|^{p-2}\frac{\de w}{\de \nu} +\beta |w|^{p-2}w=0
      &\text{on } \de B_{R},
    \end{array}
  \right.
\end{equation}
where $\lambda_{1,\mathcal E}(B_{R})$ denotes the first eigenvalue. 
It is known (see, for example, \cite{daifu}) that problem \eqref{eq:iso3} admits a positive radially decreasing solution $w_{p}(x)=\varrho_{p}(|x|)$, $0\le |x|\le R$,  such that $\varrho_{p}\in C^{\infty}(0,R)\cap C^{1}([0,R])$ and verifies
	\begin{equation}
		\label{eqrad}
		\left\{
	\begin{array}{ll}
-(p-1)(-\varrho'_{p}(r))^{p-2} \varrho_{p}''(r)+\dfrac{n-1}{r}(-\varrho'_{p}(r))^{p-1}	
	=\lambda_{1,\mathcal E}(B_{R})\varrho_{p}(r)^{p-1},&r\in]0,R[,\\[.4cm]
	\varrho_{p}'(0)=0,\\[.3cm]
	-(-\varrho_{p}'(R))^{p-1} + \beta \varrho_{p}(R)^{p-1} = 0.
	\end{array}
	\right.
	\end{equation}
	Let $v_{p}(x) = \varrho_{p}(H^{o}(x))$, $x\in \mathcal W_{R}$. Using properties \eqref{eq:om}--\eqref{eq:HH0}, for $x\in\mathcal W_{R} \setminus \{0\}$ we have that
	\[
	H(D v_{p}(x))=-\varrho_{p}'(H^{o}(x))H(D H^{o}(x))= -\varrho_{p}'(H^{o}(x)),
	\]
and
\[
	DH(D v_{p}(x))=-DH(D H^{o}(x))= -\frac{x}{H^{o}(x)},
\]
which imply that
\begin{equation}
\label{Fp}
F_{p}(Dv_{p})=-(-\varrho'(H^{o}(x)))^{p-1}\frac{x}{H^{o}(x)},
\end{equation}
and then, by \eqref{eqrad},
\begin{equation}
\label{divFp}
\begin{split}
-\divergenza(F_{p}(Dv_{p}))&=-(p-1)(-\varrho'_{p}(H^{o}(x)))^{p-2} \varrho_{p}''(H^{o}(x))+\frac{n-1}{H^{o}(x)}(-\varrho'_{p}(H^{o}(x)))^{p-1}
\\ &=\lambda_{1,\mathcal E}(B_{R}) v_{p}(x)^{p-1}\qquad \text{for }x\in \mathcal W_{R}\setminus \{0\}. 
\end{split}
\end{equation}
As regards the boundary condition, observing that $\nu(x)= DH^{o}(x)/|DH^{o}(x)|$, by \eqref{Fp}, the properties \eqref{eq:om}, \eqref{eq:H1}, and \eqref{eqrad}  we have that
\begin{equation}
\label{Fpbd}
\begin{split}
F_{p}(v_{p}(x))\cdot \nu(x) +\beta H(\nu(x))\, v_{p}(x)^{p-1} &= \frac{1}{|DH^{o}(x)|}
\Big(
	-(-\varrho'(R))^{p-1}+\beta \varrho_{p}(R)^{p-1}
\Big)\\[.1cm] &=0\qquad \text{for }x\in \de\mathcal W_{R}.
\end{split}
\end{equation}
Hence, integrating \eqref{divFp} on $\mathcal W_{R}\setminus \mathcal W_{\eps}$, we can use the divergence theorem and the boundary condition \eqref{Fpbd}, and let $\eps$ going to $0$, obtaining that $v_{p}$ verifies
\begin{equation}
	\label{pbmu}
  	\left\{
  		\begin{array}{ll}
      	-\divergenza\left(F_p(Dv_{p})\right)=\lambda_{1,\mathcal E}(B_{R})
      	v_{p}^{p-1} 
      	&\text{in } \mathcal W_{R},\\[.2cm] 
      F_p(Dv)\cdot \nu +\beta H(\nu)\, v_{p}^{p-1}=0
      &\text{on } \de\mathcal W_{R}.
    \end{array}
  \right.
\end{equation}
But Theorem \ref{theodue} guarantees that a positive solution of \eqref{pbmu} has to be a first eigenfunction, and
\[
	\lambda_{1}(\mathcal W_{R})=\lambda_{1,\mathcal E}(B_{R}).
\] 
This concludes the proof.
\end{proof}
\begin{rem}
We observe that the proof of the above theorem shows that, for any convex function $H$ we can consider, the first eigenvalue in the ball $\mathcal W_{R}=\{H^{o}(x)<R\}$ is the same, and coincides with the first eigenvalue for the $p$-Laplacian problem \eqref{eq:iso3} in the Euclidean ball $B_{R}$ (with the same $R$). 
\end{rem}
Next two lemmata will be useful in the proof of the main result. Their proofs are analogous to the ones obtained in \cite{bd10}. For the sake of completeness, we write them in details.
\begin{lemma}
\label{monotonia1}
If $0<r<s$, then $\lambda_{1}(\mathcal W_{r})>\lambda_{1}(\mathcal W_{s})$. 
\end{lemma}
\begin{proof}
Let $v_{p}$ a minimizer of \eqref{minrad}, with $R=r$, and take $w(x)=v_{p}\big(\frac{r}{s} x\big)$, $x\in \mathcal W_{s}$. Then, by the homogeneity of $H$ we get
\begin{equation*}
\begin{split}
\lambda_{1}(\mathcal W_{s}) &\le \dfrac{\ds\int_{\mathcal W_{s}} [H(Dw)]^{p}dx +\beta\int_{\de \mathcal W_{s}}|w|^{p}H(\nu)d\sigma}{\ds\int_{\mathcal W_{s}} |w|^{p}dx}\\
&= \dfrac{\ds \left(\frac r s\right)^{p} \int_{\mathcal W_{r}} [H(Dv_{p})]^{p}dx +\beta\, \frac r s\int_{\de \mathcal W_{r}}|v_{p}|^{p}H(\nu)d\sigma}{\ds\int_{\mathcal W_{r}} |v_{p}|^{p}dx} \\
&< \dfrac{\ds \int_{\mathcal W_{r}} [H(Dv_{p})]^{p}dx +\beta\, \int_{\de \mathcal W_{r}} |v_{p}|^{p}H(\nu)d\sigma}
{\ds\int_{\mathcal W_{r}} |v_{p}|^{p}dx} =\lambda_{1}(W_{r})
\end{split}
\end{equation*}
\end{proof}
We stress that by \eqref{eqrad}, if $v_{p}(x)=\varrho_{p}(H^{o}(x))$ is the positive solution in ${\mathcal W}_{R}$ we found in Theorem \ref{teorad}, we have that, for $x\in \de \mathcal W_{R}$,
\[
\beta= \frac{\left[H\big(D v_{p}(x)\big)\right]^{p-1}}{{v_{p}(x)}^{p-1}}.
\]
Then, for every $0\le r\le R$, we define
\begin{equation}
\label{betar}
\beta_{r}= \frac{\left[H\big(D v_{p}(x)\big)\right]^{p-1}}{{v_{p}(x)}^{p-1}},\qquad \text{for }H^{o}(x)=r.
\end{equation}
Let us observe that $\beta_{0}=0$ and $\beta_{R}=\beta$.
\begin{lemma}
\label{lemmabeta}
If $0\le r< s \le R$, then $\beta_{r}<\beta_{s}$.
\end{lemma}
\begin{proof}
We first observe that, similarly as in the proof of Theorem \ref{teorad},  for $0<r<R$, the function $v_{p}$ is such that
\[
  \left\{
    \begin{array}{ll}
      -\divergenza\left( F_p(Dv_{p}) \right)=\lambda_{1}(\mathcal W_{R}) 
      {v}_{p}^{p-1}
      &\text{in } 
      \mathcal W_{r},\\[.2cm] 
      F_p(Dv_{p})\cdot \nu +\beta_{r} H(\nu)\,v_{p}^{p-1}=0
      &\text{on } \de\mathcal W_{r}.
    \end{array}
  \right.
\]	
Then, denoted by $\lambda_{1}(\mathcal W_{r}, \beta_{r})$ the first eigenvalue in $\mathcal W_{r}$ with $\beta=\beta_{r}$, by Theorem \ref{theodue} we have necessarily $\lambda_{1}(\mathcal W_{R})=\lambda_{1}(\mathcal W_{r}, \beta_{r})$ for all $r\in ]0,R]$. Hence, by Lemma \ref{monotonia1} we obtain, for $0<r<s\le R$, that
\begin{multline*}
\dfrac{\ds\int_{\mathcal W_{r}} [H(Dv_{p})]^{p}dx +\beta_{r}\int_{\de \mathcal W_{r}}v_{p}^{p}H(\nu)d\sigma}{\ds\int_{\mathcal W_{r}} v_{p}^{p}dx}=
\lambda_{1}(\mathcal W_{r},\beta_{r})=\lambda_{1}(\mathcal W_{s},\beta_{s}) < \lambda_{1}(\mathcal W_{r},\beta_{s})\le
\\
\le
\dfrac{\ds\int_{\mathcal W_{r}} [H(Dv_{p})]^{p}\,dx +\beta_{s}\int_{\de \mathcal W_{r}}v_{p}^{p}H(\nu)d\sigma}{\ds\int_{\mathcal W_{r}} v_{p}^{p}\,dx},
\end{multline*}
and then $\beta_{r}<\beta_{s}$. 
\end{proof}

\section{A representation formula for $\lambda_{1}(\Omega)$}
\label{rap}
Now we prove a level set representation formula for the first eigenvalue $\lambda_{1}(\Omega)$. To this aim, we will use the following notation. Let $\tilde u_{p}$ be the first positive eigenfunction such that $\max \tilde u_{p}=1$. Then, for $t\in[0,1]$,
\begin{equation*}
	\begin{array}{l}
		U_{t} =\{x\in \Omega\colon \tilde u_{p}>t\},\\[.1cm]
		S_{t} =\{x\in \Omega\colon \tilde u_{p}=t\},\\[.1cm]
		\Gamma_{t}=\{x\in \de\Omega\colon \tilde u_{p}>t\}.
	\end{array}
\end{equation*}
First of all, it is worth to observe that the anisotropic areas of the sets $\de U_{t}$, $S_{t}$ and $\Gamma_{t}$, defined in \ref{perdef}, are related in the following way. 
\begin{lemma}
\label{lemmamisure}
	There exists a countable set $\mathcal Q\subset]0,1[$ such that
	\begin{equation}
		\label{misure}
			\sigma_{H}(\de U_{t}) \le \sigma_{H} (\Gamma_{t})+\sigma_{H}
			(S_{t}), \quad \forall t\in]0,1[\setminus \mathcal Q. 
	\end{equation}
\end{lemma}
\begin{proof}
The proof follows similarly as in \cite{bd10}. The continuity up to the boundary of the eigenfunction $\utilde$, given in Theorem \ref{teoremone}, guarantees that
\[
\de U_{t} \cap \Omega \subseteq S_{t},\quad \de U_{t} \cap \de \Omega \subseteq \tilde\Gamma_{t}
\]
for any $t\in [0,1]$, where $\tilde \Gamma_{t}=\{x\in \de \Omega\colon\utilde\ge t\}$. Moreover, by \cite[Section 1.2.3]{maz} we have that
\[
\int_{0}^{\infty}\sigma_{H}(\Gamma_{t})dt = \int_{0}^{\infty} \sigma_{H}(\tilde\Gamma_{t}) dt = \int_{\de \Omega} \tilde u\, d\sigma_{H} \le 
 \sigma_{H}(\de \Omega) <+\infty.
\]
Hence $\sigma_{H} (\Gamma_{t})\le \sigma_{H}(\tilde\Gamma_{t})<+\infty$ and then $\sigma_{H} (\Gamma_{t})= \sigma_{H}(\tilde\Gamma_{t})$ for a.e. $t\in[0,1]$. Moreover, being $\sigma_{H} (\Gamma_{t})$ and $\sigma_{H} (\tilde\Gamma_{t})$ monotone decreasing in $t$, they are continuous in $[0,1]$ up to a countable set $\mathcal Q$. Hence, 
\[
\sigma_{H}(\de U_{t}) = \sigma_{H}(\de U_{t}\cap \Omega)+\sigma_{H}(\de U_{t}\cap\de \Omega)\le \sigma_{H}(S_{t})+\sigma_{H}(\Gamma_{t}),
\] 
for all $t\in [0,1]\setminus \mathcal Q$. 
\end{proof}

If we formally divide both terms in the equation in \eqref{eq:2} by $\tilde u_{p}^{p-1}$, and integrate in $U_{t}$, by \eqref{eq:om} and the boundary condition we get
\begin{multline}
\label{formal}
\lambda_{1}(\Omega) |U_{t}| =	\int_{U_{t}} \frac{-\divergenza\big(F_{p}(D\tilde u_{p})\big)}{\tilde u_{p}^{p-1}}\, dx =\\= -(p-1)\int_{U_{t}} \frac{[H(D \tilde u_{p})]^{p-1} H_{\xi}(D\tilde u_{p}) \cdot D\tilde u_{p}}{\tilde u_{p}^{p}} \,dx - \int_{\de U_{t}} 
\frac{[H(D\tilde u_{p})]^{p-1}}{\tilde u_{p}^{p-1}}\, H_{\xi}(D\tilde u_{p}) \cdot \nu\, d\sigma = \\ =
-(p-1)\int_{U_{t}} \frac{[H(D \tilde u_{p})]^{p}}{\tilde u_{p}^{p}} \,dx + \int_{ S_{t}} 
\frac{[H(D\tilde u_{p})]^{p-1}}{\tilde u_{p}^{p-1}}H(\nu)\,d\sigma +
\beta \int_{ \Gamma_{t}} H(\nu)\,d\sigma=\\= |U_{t}| \f_{\Omega}\left(U_{t}, \frac{[H(D\tilde u_{p})]^{p-1}}{\tilde u_{p}^{p-1}} \right),
\end{multline}
where
\begin{equation}
\label{formF}
\f_{\Omega}(U_{t},\varphi)=\frac{1}{|U_{t}|}\left(
-(p-1)\int_{U_{t}} \varphi^{p'} \,dx + \int_{ S_{t}} 
\varphi H(\nu)\,d\sigma +
\beta \int_{ \Gamma_{t}} H(\nu)\,d\sigma\right),
\end{equation}
with $\varphi$ nonnegative measurable function in $\Omega$. The formal computations in \eqref{formal} give a representation formula of $\lambda_{1}(\Omega)$ which will be rigorously proved in the result below.
\begin{theo}
\label{theoform}
Let $\tilde u_{p}\in C^{1,\alpha}(\Omega)\cap C(\bar\Omega)$ be the positive minimizer of \eqref{eq:1} such that $\max \tilde u_{p}=1$. Then,
for a.e. $t\in]0,1[$, 
\begin{equation}
\label{tesitest}
\lambda_{1}(\Omega)=\f_{\Omega}\left(U_{t}, \frac{[H(D\tilde u_{p})]^{p-1}}{\tilde u_{p}^{p-1}} \right).
\end{equation}
\end{theo}
\begin{proof}
Let $0<\eps<t<1$, and 
\[
\psi_{\eps}= 
\begin{cases}
0 &\text{if }\tilde u_{p} \le t\\[.2cm]
\dfrac{u-t}{\eps} \dfrac{1}{\tilde u_{p}^{p-1}} &\text{if }t<\tilde u_{p}<t+\eps \\[.2cm]
\dfrac{1}{\tilde u_{p}^{p-1}}&\text{if }\tilde u_{p}\ge t+\eps.
\end{cases}
\]
The functions $\psi_{\eps}$ are in $W^{1,p}(\Omega)$ and increasingly converge to $\tilde u_{p}^{-(p-1)}\chi_{U_{t}}$ as $\eps \searrow 0$. Moreover,
\[
D \psi_{\eps}= 
\begin{cases}
0 & \text{if }\tilde u_{p} < t\\[.3cm]
\dfrac{1}{\eps}\left((p-1)\dfrac{t}{\tilde u_{p}}+2 - p \right) \dfrac{D \tilde u_{p}}{\tilde u_{p}^{p-1}} &\text{if }t<\title u_{p}<t+\eps \\[.4cm]
-(p-1)\dfrac{D\tilde u_{p}}{\tilde u_{p}^{p}}&\text{if }\tilde u_{p}> t+\eps.
\end{cases}
\]
Then, choosing $\psi_{\eps}$ as test function in \eqref{defsol}, we get that the first integral is
\begin{multline*}
-(p-1)\int_{U_{t+\eps}} \frac{[H(D\tilde u_{p})]^{p}}{\tilde u_{p}^{p}} dx +
\frac 1 \eps \int_{U_{t}\setminus U_{t+\eps}} \frac{[H(D\tilde u_{p})]^{p}}{\utilde ^{p-1}}\left( (p-1)\frac{t}{\utilde}+2-p \right)dx = \\ =
-(p-1)\int_{U_{t+\eps}} \frac{[H(D\tilde u_{p})]^{p}}{\tilde u_{p}^{p}} dx + 
\frac 1 \eps \int_{t}^{t+\eps} \left( (p-1)\frac{t}{\tau}+2-p \right) \int_{S_{\tau}} \frac{[H(D\tilde u_{p})]^{p-1}}{\utilde ^{p-1}} H(\nu) d\sigma,
\end{multline*}
where last equality follows by the coarea formula. 
Then, reasoning similarly as in \cite{bd10}, we get that
\[
	\int_{\Omega} [H(D\utilde)]^{p-1}H_{\xi}(D\utilde) \cdot D\psi_{\eps} dx  	\xrightarrow{\eps\rightarrow 0} -(p-1) \int_{U_{t}} \frac{[H(D\utilde)]^{p}}{\utilde^{p}} dx + \int_{S_{t}} \frac{[H(D\tilde u_{p})]^{p-1}}{\utilde ^{p-1}} H(\nu) d\sigma.
\]
As regards the other two integrals in \eqref{defsol}, we have that
\begin{equation*}
\beta\int_{\de\Omega}\utilde^{p-1}\psi_{\eps}H(\nu)\,d\sigma =
\beta\int_{\Gamma_{t+\eps}} H(\nu)\,d\sigma+ 
\beta\int_{\Gamma_{t}\setminus\Gamma_{t+\eps}}\frac{u-t}{\eps} H(\nu)\,d\sigma
\xlongrightarrow{\eps\rightarrow 0} \beta\int_{\Gamma_{t}}H(\nu)\,d\sigma,
\end{equation*}
and, by monotone convergence theorem and the definition of $\psi_{\eps}$,
\[
\lambda_{1}(\Omega)\int_{\Omega} \utilde^{p-1} \psi_{\eps}dx \;\;\xrightarrow{\eps\rightarrow 0}\;\; \lambda_{1}(\Omega)|U_{t}|.
\]
Summing the three limits, we get \eqref{tesitest}.
\end{proof}

\begin{theo}
\label{teotesithm}
Let $\varphi$ be a nonnegative function in $\Omega$ such that $\varphi\in L^{p'}(\Omega)$. If $\varphi\not\equiv [H(D \utilde)]^{p-1}/\utilde^{p-1}$, where $\utilde$ is the eigenfunction given in Theorem \ref{theoform}, and $\mathcal F_{\Omega}$ is the functional defined in \eqref{formF}, then there exists a set $S\subset ]0,1[$ with positive measure such that for every $t\in S$ it holds that
\begin{equation}
	\label{tesithm}
	\lambda_{1}(\Omega)>\mathcal F_{\Omega}(U_{t},\varphi).
\end{equation}
\end{theo}
\begin{proof}
	The proof is similar to the one obtained in \cite{bd10}, and we only sketch it here. It can be divided in two main steps. First, we claim that, if
	\[
	w(x):=\varphi-\frac{[H(D \utilde)]^{p-1}}{\utilde^{p-1}},\qquad I(t):=				\int_{U_{t}} w\frac{H(D \utilde)}{\utilde}\,dx,
	\] 
	then $I\colon ]0,1[\rightarrow \mathbb R$ is locally absolutely continuous and
	\begin{equation}
	\label{firstineq}
	\mathcal F_{\Omega}(U_{t},\varphi)\le\lambda_{1}(\Omega)-\frac{1}{|U_{t}|t^{p-1}} \Big(\frac{d}{dt}t^{p} I(t) \Big),
	\end{equation}
	for almost every $t\in ]0,1[$. Second, we show that the derivative in \eqref{firstineq} is strictly positive in a subset of $]0,1[$ of positive measure. 

In order to prove \eqref{firstineq}, writing the representation formula \eqref{tesitest} in terms of $w$, it follows that, for a.e. $t\in]0,1[$,
\begin{equation}
\label{secondineq}
\begin{split}
\mathcal F_{\Omega}(U_{t},\varphi)&= \lambda_{1}(\Omega)+\frac{1}{|U_{t}|}
\left(
	\int_{S_{t}} wH(\nu)\,d\sigma-(p-1)\int_{U_{t}}
	\Big(\varphi^{p'}-\frac{[H(D\utilde)]^{p}}{\utilde^{p}}\Big)dx
\right)\\
&\le 
\lambda_{1}(\Omega)+\frac{1}{|U_{t}|}
\left(
	\int_{S_{t}} wH(\nu)\,d\sigma-p\int_{U_{t}}
	 w\frac{H(D\utilde)}{\utilde}dx
\right)\\
&=
\lambda_{1}(\Omega)+\frac{1}{|U_{t}|}
\left(
	\int_{S_{t}} wH(\nu)\,d\sigma-p\, I(t)
\right)
\end{split}
\end{equation}
where the inequality in \eqref{secondineq} follows from the inequality $\varphi^{p'}\ge v^{p'}+p'v^{p'-1}(\varphi-v)$, with $\varphi,v\ge 0$. Applying the coarea formula, it is possible to rewrite $I(t)$ as
\[
	I(t)= \int_{U_{t}}w \frac{H(D\utilde)}{\utilde}\,dx = 
	\int_{t}^{1}\frac{1}{\tau}d\tau \int_{S_{\tau}} w\,H(\nu) \,d\sigma.
\]
This assures that $I(t)$ is locally absolutely continuous in $]0,1[$ and, for almost every $t\in]0,1[$ we have
\[
-\frac{d}{dt}\big( t^{p}I(t) \big)=t^{p-1} \left(\int_{S_{t}}w\,H(\nu)\,d\sigma -pI(t)\right).
\] 
Substituting in \eqref{secondineq}, the inequality \eqref{firstineq} follows. In order to conclude the proof, arguing by contradiction exactly as in \cite[Theorem 3.2]{bd10}, it is possible to show that $G(t):=t^{p}I(t)$ has positive derivative in a set of positive measure. Together with \eqref{firstineq}, this implies \eqref{tesithm}.
\end{proof}

\section{Main result}
Now we are in position to state and prove the desired Faber-Krahn inequality.
\begin{theo}
Let $\Omega\subset \mathbb R^{n}$, $n\ge 2$, be a bounded Lipschitz domain, and $H\colon\mathbb R^{n}\rightarrow [0,+\infty[$ a function with strictly convex sublevel sets which satisfies \eqref{eq:omo}, \eqref{eq:lin}, and \eqref{ipellipt}. Then, 
\begin{equation}
\label{fktesi}
 \lambda_{1}(\Omega)\ge \lambda_{1}(\mathcal W_{R}),
\end{equation}
where $\mathcal W_{R}$ is the Wulff shape centered at the origin such that $|\mathcal W_{R}|=|\Omega|$. The equality holds if and only if $\Omega$ is a Wulff shape.  
\end{theo}

\begin{proof}
The first step in order to prove the result is to construct a suitable test function in $\Omega$ for \eqref{formF}.
Let $v_{p}$ be a positive eigenfunction of the anisotropic radial problem \eqref{eq:3} in $\mathcal W_{R}$. By Theorem \ref{teorad}, $v_{p}$ is a function depending only by $H^{o}(x)$, and then we are able to define, as in \eqref{betar}, the function 
\[
	\beta_{r}=\varphi_{\star}(x) = \frac{[H(D v_{p}(x))]^{p-1}}{v_{p}(x)^{p-1}},	
	\quad\text{with }x\in \mathcal{\overline W}_{R},\text{ i.e. } H^{o}(x)=r\in [0,R].
\]

As before, let $\utilde$ be the first eigenfunction of \eqref{eq:2} in 
$\Omega$ such that $\|\utilde\|_{\infty}=1$. Using the same notation of Section \ref{rap}, for any $t\in]0,1[$ we consider $\mathcal W_{r(t)}$, the Wulff shape centered at the origin, where $r(t)$ is the positive number such that $|U_{t}|=|\mathcal W_{r(t)}|$. Then, for $x\in\Omega$ and 
$\utilde(x)=t$, we define
\[
	\varphi(x):=\beta_{r(t)}.
\]
Similarly as in \cite{bd10}, $\varphi$ is a measurable function. Thanks to this test function, we can compare $\mathcal F_{\Omega}(U_{t},\varphi)$ with $\mathcal F_{\mathcal W_{R}}(B_{r(t)},\varphi_{\star})$. Indeed, we claim that
\begin{equation}
\begin{split}
	\label{compar}
	\mathcal F_{\Omega}(U_{t},\varphi) & \ge
 \frac{1}{|\mathcal W_{r(t)}|}
	\left( -(p-1)\int_{\mathcal W_{r(t)}} \varphi_{\star}^{p'}dx + 
	\int_{\de \mathcal W_{r}} \varphi_{\star} H(\nu)d\sigma
	\right)
	 \\ &=
	\mathcal F_{\mathcal W_{R}}(\mathcal W_{r(t)},\varphi_{\star}) 
\end{split}
\end{equation}
for all $t\in ]0,1[\setminus \mathcal Q$, where $\mathcal Q$ is the set 
of Lemma \ref{lemmamisure}. 
In order to show \eqref{compar}, we first observe that by \cite[Section 1.2.3]{maz}, being $|U_{t}|=|\mathcal W_{r(t)}|$ for all $t\in ]0,1[$
\begin{equation}
\label{eqpr}
	\int_{U_{t}} \varphi^{p'}dx =\int_{\mathcal W_{r(t)}} \varphi_{\star}^{p'}dx.
\end{equation}
Moreover, the anisotropic isoperimetric inequality \eqref{isop}, Lemma \ref{lemmamisure} and being, by Lemma \ref{lemmabeta}, $\beta_{r(t)}\le \beta$ for any $t$, we have that
\begin{multline}
\label{ineqpr}
	\int_{\de \mathcal W_{r(t)}} \varphi_{\star} H(\nu)d\sigma
	= \beta_{r(t)} \sigma_{H}(\de \mathcal W_{r(t)})\le \\ \le \beta_{r(t)}\sigma_{H}(\de U_{t}) \le \beta_{r(t)}\sigma_{H}(S_{t}) + \beta_{r(t)}\sigma_{H}(\Gamma_{t})  \le \int_{S_{t}} \varphi H(\nu) d\sigma + \beta \int_{\Gamma_{t}} H(\nu)d\sigma.
\end{multline}
Hence, joining \eqref{eqpr} and \eqref{ineqpr} we get \eqref{compar}. 
Then, applying the level set representation formula \eqref{tesitest} in the anisotropic radial case, and \eqref{tesithm}, by \eqref{compar} we get
\[
\lambda_{1}(\mathcal W_{R}) = 
\mathcal F_{\mathcal W_{R}}(\mathcal W_{r(t)},\varphi_{\star}) 
\le \mathcal F_{\Omega}(U_{t},\varphi)
\le \lambda_{1}(\Omega)
\]
for some $t\in ]0,1[$, which gives \eqref{fktesi}. 

In order to conclude the proof, we study the equality case. Let us suppose that $\lambda_{1}(\Omega)=\lambda_{1}(\mathcal W_{R})$. 

We first claim that, for a.e. $t\in ]0,1[$, $U_{t}$ is homothetic to a Wulff shape. Indeed, by \eqref{tesithm} and \eqref{compar}
\[
\lambda_{1}(\mathcal W_{R}) =\lambda_{1}(\Omega) \ge \mathcal F_{\Omega}(U_{t},\varphi) \ge \mathcal F_{\mathcal W_{R}}(\mathcal W_{r(t)},\varphi_{\star}) = \lambda_{1}(\mathcal W_{R})
\]
for $t$ in a set of positive measure $S\subset]0,1[$. Then by Theorems \ref{theoform} and \ref{teotesithm} it follows necessarily that 
$\varphi =\frac{H(D\utilde)^{p-1}}{\utilde^{p-1}}$. This implies that for almost every $t\in ]0,1[$, the equality in \eqref{compar} and in \eqref{ineqpr} holds. In particular, $\sigma_{H}(\de\mathcal W_{r(t)})=\sigma_{H}(\de U_{t})$ for a.e. $t$. By the equality case in the anisotropic isoperimetric inequality, we get the claim. Since $U_{t}$, $t\in]0,1[$ are nested sets all homothetic to Wulff shapes, it follows that also $\Omega=\bigcup_{t\in]0,1[} U_{t}$ is homothetic to a Wulff shape, up to a measure zero set. The Lipschitz assumption on the boundary of $\Omega$ guarantees that $\Omega=\mathcal W_{R}$, up to translations. 
\end{proof}

\bibliography{/Users/francescodellapietra/Documents/Biblioteca/library}{}
\bibliographystyle{abbrv}
\end{document}